\documentclass[12pt]{article}

\usepackage[utf8]{inputenc}
\usepackage[T2A]{fontenc}

\usepackage[english]{babel}

\usepackage{amsmath,amsfonts,amssymb,amsthm}

\usepackage[left=1.5cm,right=1.5cm,top=1.5cm,bottom=1.5cm]{geometry}

\usepackage[pdfencoding=unicode, psdextra, citecolor=blue, urlcolor=blue,
linkcolor=blue, colorlinks=true, bookmarksopen=true]{hyperref}

\usepackage{comment}

\newcommand{\R}{\mathbb{R}}
\newcommand{\eqdef}{\stackrel{\mathrm{def}}{=}}
\newcommand{\g}{\mathfrak{g}}
\newcommand{\id}{\mathrm{id}}
\newcommand{\A}{\mathcal{A}}
\newcommand{\ttt}{\mathfrak{t}}
\newcommand{\s}{\mathfrak{s}}
\newcommand{\const}{\mathrm{const}}
\newcommand{\pr}{\mathrm{pr}_{\g_1}}

\DeclareMathOperator{\dist}{dist}
\DeclareMathOperator{\ri}{ri}

\DeclareMathOperator{\sspan}{span}
\DeclareMathOperator{\ad}{ad}

\newtheorem{theorem}{Theorem}
\newtheorem{proposition}{Proposition}
\newtheorem{lemma}{Lemma}
\newtheorem{corollary}{Corollary}

\theoremstyle{definition}
\newtheorem{definition}{Definition}
\newtheorem{example}{Example}
\newtheorem{remark}{Remark}

\title{Existence theorem for sub-Lorentzian problems}

\author{
\begin{tabular}{cc}
L.\,V.~Lokutsievskiy & A.\,V.~Podobryaev\\
Steklov Mathematical Institute & Ailamazyan Program Systems Institute\\
of Russian Academy of Sciences & of Russian Academy of Sciences,\\
Moscow, Russia & Pereslavl-Zalesskiy, Russia\\
\tt{lion.lokut@gmail.com} & \tt{alex@alex.botik.ru}\\
\end{tabular}
}
\date{}

\begin{document}

\maketitle

\begin{abstract}
	
	In this paper, we prove the existence theorem for longest paths in sub-Lorentzian problems, which generalizes the classical theorem for globally hyperbolic Lorentzian manifolds. We specifically address the case of invariant structures on homogeneous spaces, as the conditions for the existence theorem in this case can be significantly simplified. In particular, it turns out that longest paths exist for any left-invariant sub-Lorentzian structures on Carnot groups.
	


\end{abstract}

\begin{footnotesize}
\textbf{Acknowledgments.}
The work of A.\,V.~Podobryaev (Sections~4--6) was supported by the Russian Science Foundation under grant 22-11-00140 (https://rscf.ru/en/project/22-11-00140/) and performed in Ailamazyan Program Systems Institute of Russian Academy of Sciences.
\end{footnotesize}

\section{Problem statement}

The most well-known example of a Lorentzian manifold is Minkowski spacetime $\mathbb{R}^{1,3}$ with the Lorentzian metric $x_0^2 - x_1^2 - x_2^2 - x_3^2$. A Lorentzian manifold is a classical generalization of Minkowski spacetime to the general case of smooth manifolds. Let's provide the precise definition.

Let $M$, $\dim M=n+1$, be a smooth manifold equipped with a sign-changing non-degenerate quadratic form $g$ of signature $(1,n)$ at each point on $M$, which smoothly dependents on the point on $M$. Tangent vectors in $TM$ with negative length are referred to as \textit{spacelike}, those with positive length are \textit{timelike}, and those with zero length are \textit{lightlike}. We will also say that a vector is \emph{non-spacelike} if it is timelike or lightlike.

Note that since the form $g$ has exactly one negative eigenvalue at each point $x\in M$, the set of all non-spacelike vectors in the tangent space $T_xM$ is the union of two symmetric convex circular cones, similar to a cone
\[
	\xi_0 \ge \sqrt{\xi_1^2+\ldots+\xi_n^2}.
\]

In special relativity, it is customary to distinguish between these two cones: one is called the \textit{past cone} and we will denote it as $C^-$, while the other is the \textit{future cone}, denoted as $C^+$. This separation is referred to as the \textit{causal structure}. The future cone in the tangent space at each point of the manifold $M$ can be defined using a 1-form $\tau\in\Lambda^1M$ (referred to as the \textit{time orientation}), such that any non-zero vector in the hyperplane $\ker{\tau}$ is spacelike,
i.e.,
\[
	\forall\, x\in M \quad \forall\, \xi\in\ker{\tau_x} \qquad \xi\ne 0 \quad\Rightarrow\quad g_x(\xi,\xi)<0.
\]
The condition $\tau_x(\xi) \geq 0$ singles out the future cone from all non-spacelike vectors $\xi\in T_xM$, while the condition $\tau_x(\xi) \leq 0$ singles out the past cone:
\[
	C^+_x = \Big\{\xi\in T_xM\ \Big|\ g_x(\xi,\xi)\ge 0\ \text{and}\ \tau_x(\xi) \geq 0\Big\},
\]
\[
	C^-_x = \Big\{\xi\in T_xM\ \Big|\ g_x(\xi,\xi)\ge 0\ \text{and}\ \tau_x(\xi) \leq 0\Big\}.
\]

\begin{definition}
	
	The triple $(M, g, \tau)$ is called a \textit{time-oriented Lorentzian manifold}.
	
\end{definition}

\begin{remark}

	It is noteworthy that the 1-form $\tau$ itself carries neither geometric nor physical significance. What matters is that this form separates the set of non-spacelike vectors into two convex cones, $C^+_x$ and $C^-_x$. Therefore, it is often convenient to disregard the 1-form $\tau$ and characterize the past and future cones through alternative means, or to replace the 1-form $\tau$ with another 1-form $\tau'$ that defines the same past and future cones.
	
\end{remark}

Note that the concepts of non-spacelike tangent vectors and the future cone lead to the definition of paths along which material points can move on a Lorentzian manifold.

\begin{definition}
	
	A Lipschitz path $x(t)$ on $M$ is called admissible if $\dot x(t)\in C^+_{x(t)}$ for almost every $t$, meaning that the velocity along the path at almost each moment in time is non-spacelike and lies within the future cone.
	
\end{definition}

The distance on the Lorentzian manifold is defined as follows:
\[
	\dist_L(x_0,x_1) = \sup\int_0^1 \sqrt{g_{x(t)}(\dot x(t),\dot x(t))}\,dt,
\]
where the supremum is taken over all admissible paths going from the point $x_0$ to the point $x_1$:
\[
	x(0)=x_0, \quad x(1)=x_1, \quad \dot x(t)\in C^+_{x(t)}\text{ for a.e.\ }t\in[0, 1].
\]
As usual, if the supremum is taken over empty set, then we put $\sup{\varnothing}=-\infty$.
\bigskip

This article is dedicated to the following question: for given two points $x_0$ and $x_1$, does there exist an admissible path on which the mentioned supremum is attained? If such a path exists, it is referred to as the \textit{longest}. It is important to note that, in general, $\dist_{L}(x_0,x_1)\neq\dist_{L}(x_1,x_0)$. Therefore, we will always specify ``from point $x_0$ to point $x_1$'' to emphasize the direction of parameter progression along the path.

It turns out that the question of existence can be addressed in a nearly identical manner for both Lorentzian and sub-Lorentzian structures. Therefore, in this work, we present an existence theorem that applies simultaneously to both cases.

\section{Sub-Lorentzian manifolds}

A sub-Lorentzian structure on a smooth manifold $M$ was firstly introduced by M.~Grochowski~\cite{grochowski1} as a pair of a subbundle $\Delta \subset TM$
(called \emph{a distribution}) with a fiber dimension of $r+1$ and a non-degenerate quadratic form $g$ of signature $(1, r)$ on it,
smoothly depending on a manifold's point. The quadratic form $g$ determines a family of cones $C_x \subset \Delta_x \subset T_xM$ of non-spacelike tangent vectors.
Admissible curves and their length can be defined a similar way to the Lorentzian case.

We consider a wider notion of a sub-Lorentzian structure as a generalization of the Lorentzian structure in two natural directions. Firstly, the circular cone $C^+_x$ is replaced by an arbitrary non-empty convex cone in $T_xM$, and secondly, the quadratic form $g$ is replaced by an arbitrary anti-norm on this cone. It is worth noting that we also consider cases where the cone $C^+_x$ is not necessarily solid and allow it to have an empty interior.

\begin{definition}\label{def:antinorm}
	
	An \emph{anti-norm} on the cone $C^+_x$ is defined as an upper semicontinuous function $\nu_x : T_xM \rightarrow \R \sqcup \{-\infty\}$ such that
	\begin{enumerate}
		\item[(1)] $\nu_x|_{C^+_x} \ge 0$ and $\nu_x(T_xM \setminus C^+_x) = -\infty$;
		\item[(2)] $\forall\, \xi \in T_xM \ \ \forall\, \lambda > 0 \ \ \nu_x(\lambda\xi) = \lambda\nu_x(\xi)$;
		\item[(3)] $\forall\, \xi,\zeta \in T_xM \ \ \nu_x(\xi+\zeta) \geq \nu_x(\xi) + \nu_x(\zeta)$.
	\end{enumerate}
	
\end{definition}

Note that, cone $C^+_x$ must be closed due to upper semicontinuity of the antinorm. Moreover, if $\nu_x|_{C^+_x}\not\equiv0$, then $\nu_x|_{\ri{C^+_x}} > 0$ where $\ri{C^+_x}$ denotes the relative interior of the cone $C^+_x$. Indeed, if a concave function attains its minimum at a point lying in the relative interior of its domain, then this function is constant on the domain.

\begin{definition}
	
	Let $M$ be a smooth manifold, and at each point $x \in M$, let there be given a non-empty, convex, and pointed cone\footnote{Meaning the apex of the cone is $0\in T_xM$, and the cone $C^+_x$ does not contain any lines.} $C^+_x \subset T_xM$ equipped with an anti-norm $\nu_x$. We will refer to the triple $(M, C^+, \nu)$ as a \emph{sub-Lorentzian manifold}.
	
\end{definition}

In the Lorentzian case, the anti-norm $\nu$ is simply the square root of the vector length (which is non-negative on the non-spacelike cone):
\[
	\forall\, \xi\in C^+_x \quad \nu_x(\xi)\eqdef\sqrt{g_x(\xi,\xi)}.
\]
In the general case, however, the anti-norm can have any nature and is not required to be quadratic.

It is worth noting that in the definition of a sub-Lorentzian manifold, we do not impose any requirements on the continuity of the objects $C^+_x$ and $\nu_x$ with respect to the point $x$. However, it is evident that complete absence of continuity requirements is not feasible. Therefore, we have specified these requirements in a separate definition.

\begin{definition}
\label{def:regularmanifold}
	A sub-Lorentzian manifold is called \textit{regular}, if the following two conditions are satisfied.
	\begin{enumerate}
		\item[(1)] The set $\mathcal{C}^+=\bigcup\limits_{x\in M}C^+_x\subset TM$ is closed.
		\item[(2)] The anti-norm $\nu$ is upper semicontinuous as a function from $TM$ to $\R\sqcup\{-\infty\}$.
	\end{enumerate}
\end{definition}

\begin{remark}
\label{rem:reg}
	Both of these conditions are inherently satisfied if the sub-Lorentzian structure is constructed as follows. Let $C^+ \subset \R^m$ be a closed pointed cone, and $\nu$ be a closed anti-norm on it. Suppose $\Phi: M \times \R^m \to TM$ is a continuous map that is fiberwise linear and non-degenerate in each fiber, i.e., $\Phi(x) \in \mathrm{Lin}(\R^m \to T_xM)$ and $\det \Phi(x) \neq 0$. Then, the following sub-Lorentzian structure with $C^+_x = \Phi(x)(C^+)$ and $\nu_x(\xi) = \nu(\Phi(x)^{-1}(\xi))$ is regular. A typical example of such a structure is an invariant structure on a homogeneous space $M$ of a Lie group $G$. To define such a structure, it is sufficient to specify a closed pointed cone $C^+_{x_0} \subset T_{x_0}M$ and an anti-norm $\nu_{x_0}$ on it, both invariant under the stabilizer $G_{x_0} = \{g \in G \,|\, g(x_0) = x_0\}$ of the point $x_0$. In particular, a left-invariant sub-Lorentzian structure on a Lie group is regular.

\end{remark}

The length on a sub-Lorentzian manifold is defined in a similar manner
\[
	\dist_{SL}(x_0,x_1) = \sup\int_0^1 \nu_{x(t)}(\dot x(t))\,dt,
\]
where the supremum is taken over all admissible paths going from the point $x_0$ to the point $x_1$:
\[
	x(0)=x_0, \quad x(1)=x_1, \quad \dot x(t)\in C^+_{x(t)}\text{ for a.e.\ }t\in[0, 1].
\]

\section{Existence Theorem}

To simplify the formulation, let us assume that $M$ is equipped with a Riemannian manifold structure (for instance, if $M=\R^{n+1}$, the standard Euclidean structure suffices). We will denote the lengths of vectors $\xi \in T_xM$ in this structure as $|\xi|$, and the distance between points $x, y \in M$ as $\dist_R(x, y)$.

\begin{theorem}\label{th:exist}
	
	Let $(M,C^+,\nu)$ be a regular sub-Lorentzian manifold, and $M$ is a complete Riemannian manifold. Let $A \in M$ be any fixed point. Suppose there exists a 1-form $\tau \in \Lambda^1(M)$ \emph{(}time orientation\emph{)} such that
	\begin{enumerate}
		
		\item[\emph{(1)}] $\forall x\in M$ $\forall \xi\in C^+_x$ we have\footnote{In particular, $\tau_x(\xi)>0$ for all $0\ne\xi\in C^+_x$.} $\tau_x(\xi)\ge |\xi|/(1+\dist_R(A,x))$,
		\item[\emph{(2)}] $d\tau = 0$.

	\end{enumerate}
	Assume that $H^1(M) = 0$. Then there exists a longest path going from a point $x_0 \in M$ to a point $x_1 \in M$ if and only if there is at least one admissible path from $x_0$ to $x_1$.
	
\end{theorem}

Before proceeding with the proof of Theorem~\ref{th:exist}, let's provide an example of its application.

\begin{example}\label{ex:hyperbolicplane}
   	Let's consider left-invariant Lorentzian structures on the semidirect product $G = \R \ltimes \R_+$, where $\R_+ = \{x \in \R \,|\, x > 0\}$. This group can be identified with the hyperbolic Lobachevsky plane. Such structures have been studied in works~\cite{sachkov-l1,sachkov-l2}, see also~\cite[\S~3.2]{sachkov}. It turns out that for some of these structures, the question of the existence of longest paths can be resolved using Theorem~\ref{th:exist}, without resorting to the study of extremal trajectories and reachable sets, as was done in the mentioned works.

   Consider the standard Lobachevsky structure $(dx^2+dy^2)/y^2$ (which is left-invariant) as the Riemannian metric on the group $G$. The left-invariant sub-Lorentzian structure is defined by a cone $C^+_{\id} \subset T_{\id}G$ in the tangent space to the identity $\id = (0, 1) \in G$. The group multiplication law in $G$ is given by
    \[
        (x_1,y_1) \cdot (x_2,y_2) = (x_1 + y_1x_2, y_1y_2), \qquad x_1,x_2 \in \R, \qquad y_1,y_2 \in \R_+.
    \]
    Then, $C^+_{(x,y)} = \{y\xi \,|\, \xi \in C^+_{\id}\}$. Choose a covector $\tau_{\id} = (a, b) \in T_{\id}^*G$. By multiplying it by a suitable number, we can assume that $\frac{|\xi|}{\tau_{\id}(\xi)} < 1$ for $\xi \in C^+_{\id}$, where $|\xi|$ is the length of $\xi$ with respect to the Riemannian structure. When this covector is spread by left shifts, we obtain the 1-form $\tau = \frac{1}{y}\left(adx + bdy\right)$. For $y\xi \in C^+_{(x,y)}$, we have
    \[
        \frac{|y\xi|}{\tau_{(x,y)}(y\xi)} = \frac{|\xi|}{\tau_{\id}(\xi)} < 1,
    \]
    wherefrom it follows that condition (1) of Theorem~\ref{th:exist} is satisfied. Condition (2) takes the form $d\tau = \frac{a}{y^2} dx \wedge dy = 0$, and it holds if and only if $a = 0$. Therefore, by Theorem~\ref{th:exist}, for any cone $C^+_{\id}$ intersecting the line $\{(*,0)\} \subset T_{\id}G$ only at the apex, there will exist longest paths on the reachable set from the point $\id$ in the corresponding Lorentzian problem.
\end{example}

\begin{proof}[Proof of Theorem~\emph{\ref{th:exist}}]
	
	Throughout the entire proof, we will actively use the Hopf--Rinow theorem, according to which, on a complete Riemannian manifold, every bounded closed set is compact. We will use this fact for both the manifolds $M$ and $TM$ (the completeness of $TM$ immediately follows from the completeness of $M$).
	
	The existence of a longest path is equivalent to the existence of an optimal solution in the following optimal control problem:
	\begin{equation}
	\label{eq:initial problem}
		\begin{gathered}
			\ell(x) = \int_0^1\nu_{x(t)}(\dot x(t))\,dt\to\max,\\
			x|_{t=0}=x_0, \quad x|_{t=1}=x_1, \quad \dot x(t)\in C^+_{x(t)}.
		\end{gathered}
	\end{equation}
	
	Let's introduce a new parameter $s$ on the path as follows:
	\[
		s(t) = \int_0^t \tau_{x(t_1)}(\dot x(t_1))\,dt_1.
	\]
	Note that $\tau_x(C^+_x) \geq 0$, so $s(t)$ is a monotonically non-decreasing function.
		
	It is important to note that $s_1 = s(1)$ does not depend on the choice of the path but only on its endpoints $x_0$ and $x_1$. Indeed, the first de Rham cohomologies are zero according to the theorem's assumption, $H^1(M) = 0$, so the integral of a closed 1-form $\tau$ over any closed curve is zero.
	
	Therefore, the aforementioned optimal control problem~\eqref{eq:initial problem} is equivalent to the following:
	\begin{equation}
	\label{eq:reparametrized problem}
		\begin{gathered}
			\ell(x) = \int_0^{s_1} \nu_{x(s)}(x'(s))\,ds\to\max,\\
			x|_{s=0}=x_0, \quad x|_{s=s_1}=x_1, \quad x'(s)\in C^+_{x(s)}, \quad \tau_{x(s)}(x'(s))=1,
		\end{gathered}
	\end{equation}
	where $x'$ denotes the derivative with respect to $s$ to avoid confusion with the derivative with respect to $t$: $x'=\frac{d}{ds}x$, $\dot x=\frac{d}{dt}x$.
	
	It is convenient to express the last two velocity constraints on $x'(s)$ in the problem~\eqref{eq:reparametrized problem} as a single differential inclusion:
	\[
		x' \in U_x = \{\xi\in C^+_x \,|\, \tau_x(\xi)=1\}.
	\]
	
	\begin{lemma}\label{lem:compactsection}
		For each $x\in M$, the set $U_x\subset T_xM$ is convex and compact.
	\end{lemma}

	\begin{proof}
		First, let's show that for each $x\in M$, the set $U_x\subset T_xM$ is convex and closed. Indeed, the hyperplane $\{\xi\in T_xM \,|\, \tau_x(\xi)=1\}$ is convex and closed. The cone $C^+_x$ is convex by the definition of a sub-Lorentzian manifold and closed according to point~(1) of Definition~\ref{def:regularmanifold} of a regular sub-Lorentzian manifold.
		
		Furthermore, for each $x\in M$, the set $U_x$ is bounded. Indeed, suppose the contrary, i.e., there exists a sequence $\xi_n \in U_x$ such that $|\xi_n| \rightarrow +\infty$. For the sequence $\bar{\xi}_n = \frac{\xi_n}{|\xi_n|} \in C^+_x$ we have $|\bar{\xi}_n| = 1$. Passing to a subsequence, due to the closedness of the cone $C^+_x$, we can assume that $\bar{\xi}_n \rightarrow \bar{\xi} \in C^+_x \setminus {0}$. In this case, $\tau_x(\bar{\xi}_n) = \frac{\tau_x(\xi_n)}{|\xi_n|} = \frac{1}{|\xi_n|} \rightarrow 0$, but $\tau_x(\bar{\xi}_n) \rightarrow \tau_x(\bar{\xi})$. Therefore, $\tau_x(\bar{\xi}) = 0$, which contradicts the condition $\tau_x(C^+_x \setminus {0}) > 0$.
	\end{proof}
	
	\begin{lemma}\label{lem:semicont}
		The set $U_x$ is upper semicontinuous \emph{(}in the sense of Filippov\emph{)} with respect to $x$ as a set-valued map $x\mapsto U_x$.
	\end{lemma}

	\begin{proof}
		Indeed, the graph of the set-valued map $x\mapsto U_x$ (as a subset of $TM$) is the intersection of the closed set $\mathcal{C}^+=\bigcup\limits_{x\in M}C^+_x$ and the closed set $\{\xi \,|\, \tau_x(\xi)=1\}\subset TM$. Furthermore, in the neighborhood of any point in $M$, the sets $U_x$ are bounded due to the inequality
		\begin{equation}
		\label{eq:xi estimate}
			\forall\, \xi\in U_x\qquad |\xi|\le 1+\dist_R(A,x).
		\end{equation}
		It remains to note that, according to~\cite[Ch.~2, \S\,5, Lemma~14]{Filippov}, such a set-valued map must be upper semicontinuous.
	\end{proof}

	Let us consider the following optimal control problem extended to the manifold $M \times \R$. The goal is to find a curve
$\left(x(s), z(s)\right) \in M \times \R$, where $s \in [0, s_1]$, such that
	\begin{equation}
	\label{eq:auxiliary problem}
		\begin{gathered}
			z(s_1)\to\max,\\
			x(0)=x_0, \quad z(0)=0, \quad x(s_1) = x_1,\\
			(x'(s),z'(s))\in V_{x(s)} = \{(\xi,\zeta) \,|\, \xi\in U_{x(s)},\ 0\le \zeta\le \nu_{x(s)}(\xi)\}\subset T_{x(s)}M\times\R.
		\end{gathered}
	\end{equation}
	
	The problem~\eqref{eq:auxiliary problem} is equivalent to the problem~\eqref{eq:reparametrized problem} (and, consequently, the original~\eqref{eq:initial problem}), as $z(s)$ grows no faster than $\nu_x(x'(s))$. Let's demonstrate that there exists an optimal solution for the problem~\eqref{eq:auxiliary problem}.
	
	\begin{lemma}
	
		The set $V_x$ is convex and compact for every $x\in M$ and is upper semicontinuous \emph{(}in the sense of Filippov\emph{)} as a set-valued map $x\mapsto V_x$.
	
	\end{lemma}
	
	\begin{proof}
		
		The function $\nu_x(\xi)$ is concave and closed on the convex set $U_x$, so its hypograph $\{(\xi,\zeta) \,|\, \xi\in U_x, \, a\le \nu_x(\xi)\}$ is convex and closed by definition. The set $V_x$ is the intersection of this hypograph with the half-space $\zeta\ge 0$. Thus, the set $V_x$ is convex and closed as the intersection of two convex closed sets. Moreover, the set $V_x\subset T_xM\times\R$ is bounded. Indeed, the set $U_x$ is compact, so the upper semicontinuous function $\nu_x(\xi)$ attains its maximum over $\xi\in U_x$. Let's denote it by $Z_x$. Thus, $V_x\subset U_x\times[0, Z_x]$, and the latter set is bounded as the product of two compact sets. Therefore, the set $V_x$ is compact.
		
		Now let's show that the set-valued map $x\mapsto V_x$ is upper semicontinuous. According to~\cite[Ch.~2, \S\,5, Lemma~14]{Filippov}, it suffices to verify that the graph of this set-valued map is closed. Suppose $(\xi_k,\zeta_k)$ is a sequence of elements in the graph converging to some element $(\xi,\zeta)\in TM\times\R$. We will demonstrate that the point $(\xi,\zeta)$ also belongs to the graph.
		
		By the definition of the set-valued map, $(\xi_k,\zeta_k)\in V_{x_k}$ for some points $x_k\in M$. Clearly, $x_k=\pi(\xi_k)$, where $\pi:TM\to M$ is the canonical projection onto the base. Therefore, the sequence $x_k$ is also convergent, $x_k\to x$, and $x=\pi(\xi)$, i.e., $\xi\in T_xM$. According to Lemma~\ref{lem:semicont}, $\xi \in U_x \subset T_xM$. The function $\nu:TM\to\R\sqcup\{-\infty\}$ is upper semicontinuous, so
		\[
			\nu_x(\xi) \ge \limsup_{k\to\infty}\nu_{x_k}(\xi_k) \ge \limsup_{k\to\infty}\zeta_k = \zeta,
		\]
		and we obtain $0\le\zeta\le \nu_x(\xi)$. Therefore, $(\xi,\zeta)\in V_x$, which proves the closedness of the graph.
	\end{proof}

    The reachable set from the point $(x_0, 0)$ in time less or equal to $s_1$ for the control system of problem~\eqref{eq:auxiliary problem} is compact according to Filippov's theorem (see~\cite[Theorem 3, \S7, Chapter 2]{Filippov}). Indeed, the existence of a solution on small intervals follows from Theorem 1 in \cite[\S7, Chapter 2]{Filippov}, and the extendability of the solution to any length of time interval follows from Theorem 2 in \cite[\S7, Chapter 2]{Filippov} and the condition of sublinear growth~\eqref{eq:xi estimate}. Note that Filippov proved his theorems for the case of $M=\R^n$, but his proofs can be applied verbatim by replacing $\R^n$ with any complete Riemannian manifold.
	
	It remains to note that the intersection of the reachable set from the point $(x_0, 0)$ in time not exceeding $s_1$ of the control system of problem~\eqref{eq:auxiliary problem} with the set $\{x_1\} \times \R$ is compact as well. Indeed, this intersection is non-empty (since the point $x_1$ can be reached from the point $x_0$ by an admissible path of problem~\eqref{eq:initial problem} due to the assumption) and the set $\{x_1\} \times \R$ is closed. So, the projection of this intersection to $\R$ has the greatest element $z_1$. It follows, that there exists an admissible path of problem~\eqref{eq:auxiliary problem} from the point $(x_0, 0)$ to the point $(x_1, z_1)$ which is an optimal solution.
\end{proof}

\section{Invariant sub-Lorentzian structures}
\label{sec:leftinvariant}

The conditions of the existence theorem for invariant sub-Lorentzian structures on homogeneous spaces of Lie groups can be significantly simplified. We will assume that an invariant Riemannian structure is defined on the homogeneous manifold $M$ with respect to the action of a Lie group $G$, i.e., a positively definite $G$-invariant quadratic form $g^R$. We will denote by $|\xi| = \sqrt{g^R_x(\xi, \xi)}$ the length of a tangent vector $\xi \in T_xM$ with respect to this structure.

\begin{theorem}
    \label{th:leftinv}

	Let $(M, C^+, \nu)$ be an invariant sub-Lorentzian structure on a homogeneous space $M$ with respect to an action of a Lie group $G$, where $M$ is a homogeneous Riemannian manifold. Fix a point $x_0 \in M$. Assume that the covector $\tau_{x_0} \in T^*_{x_0}M$ is invariant under the stabilizer $G_{x_0}$
and $\tau_{x_0}|_{C^+_{x_0} \setminus 0} > 0$. Define the 1-form $\tau \in \Lambda^1 M$ using the shifts of the covector $\tau_{x_0}$ as follows:
	\[
		\forall\, \xi \in T_xM \quad \tau_x(\xi) = \tau_{x_0}(g^{-1}_* \xi),
	\]
	where $g \in G$ is such that $g x_0 = x$. Suppose that $\tau$ is an exact 1-form on the reachable set $\A$ from the point $x_0$ along admissible paths. Then there exists a longest path from $x_0$ to $x_1$ if and only if there exists at least one admissible path from $x_0$ to $x_1$.

\end{theorem}

\begin{proof}

    First of all, let's note that the form $\tau$ is well-defined. Indeed, its independence of the chosen element $g \in G$ that maps the point $x_0$ to $x$ is guaranteed by the $G_{x_0}$-invariance of the covector $\tau_{x_0}$.

    Due to the exactness of the 1-form $\tau$, there exists a function $T : \A \rightarrow \R$ such that $dT = \tau$. As shown in the proof of Theorem~\ref{th:exist}, the original optimal control problem~\eqref{eq:initial problem} is equivalent to problem~\eqref{eq:reparametrized problem}. Moreover, the terminal time of the latter problem is given by
    \[
        s_1 = \int\limits_{0}^{1}{\nu_{x(t)}(\dot{x}(t))\, dt} = T(x_1).
    \]
    So it does not depend on the choice of the admissible path $x(\cdot)$ going from the point $x_0$ to $x_1$ and is determined by the end point $x_1$.

    Note that the homogeneous Riemannian metric $g^R$ is complete, see, for example, \cite[Lemma~5.4]{beem-ehrlich-easley}.

    The set $U_x$ is a shift (by the action of the group $G$) of the set $U_{x_0}$, which is compact according to Lemma~\ref{lem:compactsection}. Therefore, there exists a constant $c > 0$ such that for any $x \in M$ and any $\xi \in U_x$, we have $|\xi| < c$. Hence, the sublinear growth condition~\eqref{eq:xi estimate} holds for the Riemannian metric $\frac{1}{c^2}g^R$. Thus, condition~(1) of Theorem~\ref{th:exist} holds, and condition~(2) is satisfied automatically. The application of this theorem completes the proof.
\end{proof}

\begin{remark}\label{rem:compactstabilizer}
    It is well known that if the stabilizer $G_{x_0}$ is compact, then there exists a $G$-invariant Riemannian metric on the manifold $M$.
    It is sufficient to obtain a $G_{x_0}$-invariant positive definite quadratic form on $T_{x_0}M$
    by averaging an arbitrary positive definite quadratic form over the compact group $G_{x_0}$,
    and then to spread it using the action of the group $G$.
\end{remark}

\section{Sub-Lorentzian Structures on Carnot Groups}
\label{sec:carnot}

Consider another example of applying the main theorem of this work. We will prove the existence of the longest paths for left-invariant sub-Lorentzian structures on Carnot groups. This example is important because control systems on Carnot groups provide a nilpotent approximation for control systems~\cite{agrachev-sarychev}.

\begin{definition}\label{def:carnot}

  	The \emph{Carnot algebra} is a graded Lie algebra $\g$ such that

    \[
        \g = \g_1 \oplus \dots \oplus \g_s, \qquad [\g_1, \g_i] = \g_{i+1}, \qquad [\g_1, \g_s] = 0, \qquad \g_s \neq 0.
    \]
    In particular, the Lie algebra $\g$ is generated by the subspace $\g_1$. The number $s$ is called the \emph{step} of the Carnot algebra. The corresponding connected and simply connected Lie group is called the \emph{Carnot group}.

\end{definition}

Below we identify the Lie algebra $\g$ with the tangent space $T_{\id}G$ at the identity element of the group.

\begin{corollary}\label{crl:carnot}

	Let $G$ be a Carnot group with the graded Carnot algebra $\g = \g_1 \oplus \dots \oplus \g_s$. Consider a closed convex pointed cone $C^+ \subset \g_1$ and an antinorm $\nu$ on it. Then for the left-invariant sub-Lorentzian problem

    \begin{equation}\label{eq:carnotproblem}
        \begin{gathered}
            \int\limits_0^1{\nu(u(t))\, dt} \rightarrow \max,\\
            x(0) = x_0, \qquad x(1) = x_1, \qquad \dot{x}(t) = x(t)_* u(t), \qquad u(t) \in C^+
        \end{gathered}
    \end{equation}
	there exists a longest path from $x_0$ to $x_1$ if and only if there is at least one admissible path from $x_0$ to $x_1$.

\end{corollary}

\begin{proof}
	
	Due to the left-invariance of the problem, we can assume that $x_0 = \id$.
	
	Since the cone $C^+$ is acute, there exists a covector $\tau_0 \in \g_1^*$ such that $\tau_0(C^+ \setminus 0) > 0$.
	Indeed, otherwise, the interior of the polar cone~\cite[\S\,14]{rockafellar} $C^{\circ} \subset \g_1^*$ would be empty.
	Then, there would exist a nontrivial subspace
    \[
        (\sspan{C^{\circ}})^{\circ} = \{\xi \in \g_1 \,|\, \forall\, \tau \in C^{\circ} \ \tau(\xi) = 0\} \subset C^+,
    \]
    which contradicts the acuteness of the cone $C^+$.

	We may assume that $\tau_0 \in \g^*$.
	To achieve this, we extend the linear function $\tau_0 : \g_1 \rightarrow \R$ to a linear function on the entire Lie algebra $\g$ by setting $\tau_0|_{[\g, \g]} = 0$.
	Consider the left-invariant 1-form on the group $G$ obtained by left shifts of the covector $\tau_0$.
	We will denote this 1-form by the symbol $\tau$.
	Let's show that the form $\tau$ is exact on the reachable set $\A$ of the system~\eqref{eq:carnotproblem} from the group identity.

    Since the group $G$ is nilpotent, the exponential map $\exp : \g \rightarrow G$ is a diffeomorphism~\cite[\S\,6.4]{kirillov}.
    Let's calculate the images of admissible velocities under the diffeomorphism $\exp^{-1}$.
    Using the Baker-Campbell-Hausdorff-Dynkin formula~\cite{dynkin} (also see, for example, \cite[\S\,6.4, Theorem~2]{kirillov}), for any $\xi, u \in \g$, due to the nilpotency of the Lie algebra $\g$, we have
    \begin{equation}\label{eq:dynkin}
        \exp{(\xi)} \cdot \exp{(s u)} = \exp{\left(\xi + s u + s \left(c_1\ad{\xi}+\dots+c_{s-1}(\ad{\xi})^{s-1}\right)u + o(s)\right)},
    \end{equation}
	where $c_1,\dots,c_{s-1}$ are certain coefficients.

    Let $x = \exp{(\xi)} \in G$, and $\gamma(s) = x \cdot \exp(s u)$ be a curve such that $\gamma(0) = x$, $\dot{\gamma}(0) = x_* u$. Let's calculate the tangent vector at the point $\xi$ of its preimage under the diffeomorphism $\exp$:
    \[
        \left.\frac{d}{ds}\right|_{s = 0} \exp^{-1}{\left(\gamma(s)\right)} = u + \left(c_1\ad{\xi}+\dots+c_{s-1}(\ad{\xi})^{s-1}\right)u.
    \]
    Then, for the velocity of the preimage $\xi(t) = \exp^{-1}{(x(t))}$ of an admissible path $x : [0, 1] \rightarrow G$ with control $u(t)$ at the Lebesgue point, we obtain:
    \begin{equation}\label{eq:velosity}
        \dot{\xi}(t) \in u(t) + [\g,\g].
    \end{equation}

    Let's define the value of the function $T : \A \rightarrow \R$ at a point $x_1 \in \A$ as the integral of the form $\tau$ along an admissible path $x : [0, 1] \rightarrow G$ from the identity element $\id$ to the point $x_1$. Due to the left-invariance of the 1-form $\tau$, we have:
    \[
        T(x_1) = \int\limits_0^{1}{\tau(\dot{x}(t)) \, dt} = \int\limits_0^{1}{\tau_0(u(t))\, dt}.
    \]

	Introduce the notation $\xi^{(1)} = \pr \circ \exp^{-1} (x)$ for the first-layer component of the point $x \in G$, where $\pr : \g \rightarrow \g_1$ is the projection along the subspace $[\g, \g]$. Then, due to the inclusion~\eqref{eq:velosity}, we have $\dot{\xi}^{(1)}(t) = u(t)$ for an admissible path $x(\cdot)$. Hence, $T(x_1) = \tau_0(\xi^{(1)}(1)) = \tau_0(\xi(1)) = \tau_0(\exp^{-1}{(x_1)})$. Therefore, the function $T$ is well-defined, and $dT = \tau$.

    Thus, the form $\tau$ is exact on the set $\A$. Taking into account Remark~\ref{rem:compactstabilizer}, we can now apply Theorem~\ref{th:leftinv}, which completes the proof.
\end{proof}

\begin{remark}\label{rem:carnot}

	The sub-Lorentzian structure given by a quadratic form with signature $(1,r)$ on a Carnot group of step $1$ is simply the Lorentzian structure on Minkowski space, i.e., a model of special relativity. For sub-Lorentzian structures on certain Carnot groups of step $2$, the following interpretation can be proposed.
	
	Consider Minkowski space $\g_1 = \R^{1,r}$ equipped with a quadratic form of signature $(1,r)$, which, in the basis $e_0, e_1, \dots, e_r$, has the form:
	\[
	    g(x) = x_0^2 - x_1^2 - \dots - x_r^2.
	\]
	Let $C^+ = \{x \in V \,|\, g(x) \geq 0, \ x_0 \geq 0\}$, and $\nu(x) = \sqrt{g(x)}$ for $x \in C^+$.
	Then, $x_0$ represents the time coordinate, while $x_1, \dots, x_r$ represent spatial coordinates. We have the decomposition $\g_1 = \ttt \oplus \s$, where $\ttt = \sspan{\{e_0\}}$, and $\s = \sspan{\{e_1,\dots,e_r\}}$.
	
	Consider the graded Carnot algebra $\g = \g_1 \oplus \g_2$ of step 2, where $\g_2 = \ttt \wedge \s$.
	Choose coordinates $y_1, \dots, y_r$ in the space $\g_2$, associated with the basis $e_0 \wedge e_1, \dots, e_0 \wedge e_r$.
	Then, the non-zero Lie brackets in this algebra take the coordinate form $[a, b]_i = a_0b_i - b_0a_i$, where $a, b \in \g_1$.
	
	The corresponding connected and simply connected Lie group $G$ is diffeomorphic to its Lie algebra.
	Consider the sub-Lorentzian structure on the group $G$ obtained by left shifts of the cone $C^+$ and the anti-norm $\nu$.
	Then, from formula~\eqref{eq:dynkin}, where $c_1 = \frac{1}{2}$, it follows that the control system of problem~\eqref{eq:carnotproblem} takes the form:
	\begin{equation}\label{eq-2step-Carnot}
	    \begin{array}{ll}
	        \dot{x}_i = u_i, & i=0,\dots,r, \\
	        \dot{y}_i = \frac{1}{2}\left(x_0 u_i - x_i u_0\right) , & i=1,\dots,r. \\
	    \end{array}
	\end{equation}
	It is easy to show that for any $i = 1,\dots,r$, the quantity $y_i(1)$ is the oriented area $S_i$ of the planar region $\Omega_i \subset \sspan{\{e_0, e_i\}}$, enclosed between the segment $I_i$ connecting the points $(0,0)$ and $(x_0(1), x_i(1))$, and the projection of the admissible path $\gamma_i(t) = (x_0(t), x_i(t))$ onto the plane $\sspan{\{e_0, e_i\}}$. Indeed, since $\partial\,\Omega_i = \gamma_i \cup I_i$, according to the Stokes theorem:
	\[
	    S_i = \int\limits_{\Omega_i}{dx_0 \wedge dx_i} = \frac{1}{2} \int\limits_{\gamma_i \cup I_i}{\left(x_0dx_i - x_i dx_0\right)}
	    = \frac{1}{2} \int\limits_{\gamma_i}{\left(x_0dx_i - x_i dx_0\right)} + \const.
	\]
	From here, taking into account~\eqref{eq-2step-Carnot}, we obtain $\dot{S}_i = \dot{y}_i$. This fact is absolutely analogous to the interpretation of the Dido inverse problem as a sub-Riemannian problem on the Heisenberg group, see~\cite[\S\,1]{montgomery} for more details.
	
	The segment connecting points $0$ and $x(1)$ represents the graph of uniform rectilinear motion. This segment, along with the velocities of any other admissible paths, lies in the cone $C^+$. Therefore, the area $S_i$ can be interpreted as the average deviation of the motion of the material point $x(t)$ from uniform rectilinear motion, with respect to the coordinate $x_i$. On the other hand, according to special relativity, the Lorentzian length of an admissible path is interpreted as the intrinsic time of the moving material point. Thus, it can be said that the sub-Lorentzian problem on the group $G$ (i.e., the slow motion problem for the control system~\eqref{eq-2step-Carnot}) consists in maximizing the intrinsic time of a material point moving from point $0$ to the given point $x(1)$, with the given $y_i(1)$, $i=1,\dots,r$ --- average per-coordinate deviations from uniform rectilinear motion.
\end{remark}

\section{Discussion of the Result}
\label{sec-discussion}

Theorem~\ref{th:exist} generalizes the well-known existence theorem for longest paths in Lorentzian geometry~\cite{beem-ehrlich-easley},
extended in works~\cite{grochowski,sachkov} to the sub-Lorentzian case, as stated in Theorem~\ref{th:compactcondition} below.
Here, by a sub-Lorentzian structure, we mean a structure defined by a distribution $\Delta \subset TM$ with a fiber dimension of $r+1$, a non-degenerate quadratic form $g$ of signature $(1,r)$ on it,
smoothly depending on the manifold's point,
and a time orientation 1-form $\tau \in \Lambda^1M$, i.e.,
\[
    C^+_x = \{\xi \in \Delta_x \,|\, g_x(\xi) \geq 0\ \text{and} \ \tau(\xi) \geq 0\}, \qquad \nu_x(\xi) = \sqrt{g_x(\xi)}, \ \xi \in C^+_x.
\]

Let's denote by $\A^+_{x_0} \subset M$ the reachable set from the point $x_0$, and by $\A^-_{x_1} \subset M$ the set of points from which the point $x_1$ is reachable.

\begin{theorem}[\cite{beem-ehrlich-easley,grochowski,sachkov}]\label{th:compactcondition}
	Let $M$ be a sub-Lorentzian manifold, $x_0, x_1 \in M$. Suppose that $x_1 \in \A^+_{x_0}$, the set $\A^+_{x_0} \cap \A^-_{x_1}$ is compact, and there are no closed admissible paths on the manifold $M$. Then there exists a longest path going from the point $x_0$ to the point $x_1$.

\end{theorem}

\begin{remark}
    The classical formulation~\cite{beem-ehrlich-easley} of Theorem~\ref{th:compactcondition} requires the absence of closed admissible paths and the compactness of $\A^+_{x_0} \cap \A^-_{x_1}$ for any points $x_0, x_1 \in M$. These conditions are known as the \emph{global hyperbolicity} of the manifold.
\end{remark}

The novelty of the Theorem~\ref{th:exist} proved in this work lies in the requirement of the closedness of the time form $\tau$. It is worth noting again that to satisfy the conditions of Theorem~\ref{th:exist}, it is necessary to find a 1-form $\tau$ such that\footnote{and condition (1) of the growth theorem~\ref{th:exist} is satisfied}
\[
	\tau|_{C^+_x \setminus 0} > 0, \qquad d\tau = 0.
\]
If such a form $\tau$ exists, then the classical theorem~\ref{th:compactcondition} is a direct consequence of Theorem~\ref{th:exist}.

\begin{proposition}\label{prop:impl}

	Let a sub-Lorentzian structure on the manifold $M$ be given by a distribution $\Delta \subset TM$, a non-degenerate quadratic form $g$ of signature $(1, r)$ on it, smoothly depending on the manifold's points, and a time direction 1-form $\tau \in \Lambda^1M$, i.e.,
    \[
        C^+_x = \{\xi \in \Delta_x \,|\, g_x(\xi) \geq 0 \ \text{\emph{and}} \ \tau(\xi) \geq 0\}, \qquad \nu_x(\xi) = \sqrt{g_x(\xi)}, \ \xi \in C^+_x.
    \]

    Then, under the conditions of Theorem \emph{\ref{th:exist}}, the conditions of Theorem \emph{\ref{th:compactcondition}} are satisfied.

\end{proposition}

\begin{proof}
    The integral of the 1-form $\tau$ along an admissible path of problem~\eqref{eq:reparametrized problem} is equal to the time of motion along that trajectory.
    Closed admissible trajectories are not possible, since according to the conditions of Theorem~\ref{th:exist} the form $\tau$ is exact. Furthermore, suppose the point $x_1$ is reachable in time $s_1$ from the point $x_0$. The set $\A^+_{x_0} \cap \A^-_{x_1}$ is the projection of the intersection of the reachable set of system~\eqref{eq:auxiliary problem} from the point $(x_0, 0)$ for a time not exceeding $s_1$ and the set of points from which the point $(x_1, z_1)$ is reachable, which are compact as it was shown in the proof of Theorem~\ref{th:exist}. Hence, the set $\A^+_{x_0} \cap \A^-_{x_1}$ is compact as well.
\end{proof}

The explicit description of reachable sets for control systems within a cone can be a computationally challenging task. For instance, this has been accomplished for certain cones on Carnot groups of low rank and step in works such as~\cite{abels-vinberg} (using some algebraic considerations) and~\cite{ardentov-ledonne-sachkov, podobryaev}
(using the Pontryagin maximum principle). In several cases (see Example~\ref{ex:hyperbolicplane} and Corollary~\ref{crl:carnot}), the application of Theorems~\ref{th:exist}--\ref{th:leftinv} allows us to infer the existence of longest paths without having information about the reachable sets.


\begin{thebibliography}{99}

    \bibitem{grochowski1}
    M.~Grochowski. On the Heisenberg sub-Lorentzian metric on $\R^3$. Geometric singularity theory. Banach Center publications. 65, 57--65. Institute of Mathematics. Polish Academy of Sciences. Warszawa (2004)
	
    \bibitem{sachkov-l1}
    Yu.\,L.~Sachkov. Lorentzian Geometry on the Lobachevsky Plane. Mathematical Notes. 114, 1, 127--130 (2023)

    \bibitem{sachkov-l2}
    Yu.~Sachkov. Lorentzian distance on the Lobachevsky plane. arXiv:2307.07706 (2023)

    \bibitem{sachkov}
    Yu.\,L.~Sachkov. Existence of the longest sub-Lorentzian paths. Differential equations. 59, 12 [to appear] (2023)

	\bibitem{Filippov}
    A.\,F.~Filippov. Differential Equations with Discontinuous Righthand Sides. Kluwer Academic Publishers, Dordrecht--Boston--London (1988)

    \bibitem{agrachev-sarychev}
    A.\,A.~Agrachev, A.\,V.~Sarychev. Filtrations of a Lie algebra of vector fields and the nilpotent approximation of controllable systems. Dokl. Math. 36, 1, 104--108 (1988)

    \bibitem{rockafellar}
    R.~Rockafellar. Convex analysis. Princeton University Press, Princeton, New Jersey (1970)

    \bibitem{kirillov}
    A.\,A.~Kirillov. Elements of the Theory of Representations. Springer-Verlag, Berlin--Heidelberg--New York (1976)

    \bibitem{dynkin}
    E.\,B.~Dynkin. Calculation of the coefficients in the Campbell–Hausdorff formula. Doklady Akademii Nauk SSSR [in Russian]. 57, 323--326 (1947)

    \bibitem{montgomery}
    R.~Montgomery. A tour of subriemannian geometries, their geodesics and applications. Mathematical surveys and monographs. 91. AMS (2002)

    \bibitem{beem-ehrlich-easley}
    J.\,K.~Beem, P.\,E.~Ehrlich, K.\,L.~Easley. Global Lorentzian Geometry. Monographs Textbooks Pure Appl. Math. 202. Marcel Dekker Inc. (1996)

    \bibitem{grochowski}
    M.~Grochowski. Geodesics in the sub-Lorentzian geometry. Bull. Polish. Acad. Sci. Math. 50 (2002)

    \bibitem{abels-vinberg}
    H.~Abels, \`{E}.\,B.~Vinberg. On free two-step nilpotent Lie semigroups and inequalities between random variables. J. Lie Theory. 29, 1, 79--87 (2019)

    \bibitem{ardentov-ledonne-sachkov}
    A.\,A.~Ardentov, E.~Le\,Donne, Yu.\,L.~Sachkov. A Sub-Finsler Problem on the Cartan Group. Proceedings of the Steklov Institute of Mathematics. 304, 42--59 (2019)

    \bibitem{podobryaev}
    A.~Podobryaev. Attainable Set for Rank 3 Step 2 Free Carnot Group with Positive Controls. 16th International Conference on Stability and Oscillations of Nonlinear Control Systems (Pyatnitskiy's Conference, IEEE. doi: 10.1109/STAB54858.2022.9807600 (2022)

\end{thebibliography}
\end{document}